\documentclass[letterpaper,12pt]{amsart}
\usepackage{amsfonts,amsmath,amssymb,amsthm}
\usepackage{multirow}
\newcommand{\N}{\ensuremath{\mathbb{N}}}
\newcommand{\R}{\ensuremath{\mathbb{R}}}

\theoremstyle{definition}
	\newtheorem{defi}{Definition}
\theoremstyle{plain}
	\newtheorem{lem}[defi]{Lemma}
	\newtheorem{thm}[defi]{Theorem}
	
	\newtheorem{claim}[defi]{Claim}

\newtheoremstyle{named}{}{}{\itshape}{}{\bfseries}{.}{.5em}{\thmnote{#3's }#1}
\theoremstyle{named}
\newtheorem*{namedtheorem}{Theorem}

\title[The finite $\mathrm{FIN}_k$ theorem]{Finite forms of Gowers' theorem on the oscillation stability of $c_0$}
\author{Diana Ojeda-Aristizabal}
\address{Department of Mathematics, Cornell University, Ithaca,NY, 14853-4201}
\thanks{The research of the author presented in this paper was partially supported by NSF
grants DMS\textendash 0757507 and DMS\textendash 1262019. Any opinions, findings, and conclusions or recommendations
expressed in this article are those of the author and do not necessarily reflect the
views of the National Science Foundation.}

\begin{document}
\bibliographystyle{plain}

	\begin{abstract}
 	We give a constructive proof of the finite version of Gowers' $\mathrm{FIN}_k$ Theorem and analyse the corresponding upper bounds. The $\mathrm{FIN}_k$ Theorem is closely related to the oscillation stability of $c_0$. The stabilization of Lipschitz functions on arbitrary finite dimensional Banach spaces was studied well before by V. Milman (see \cite[p.6]{Mil_Schech}). We compare the finite $\mathrm{FIN}_k$ Theorem with the finite stabilization principle in the case of spaces of the form $\ell_{\infty}^n$, $n\in\N$ and establish a much slower growing upper bound for the finite stabilization principle in this particular case.
	\end{abstract}
	\maketitle

\section{Introduction}
It was observed by Milman (see \cite[p.6]{Mil_Schech}) that given a real-valued Lipschitz function defined on the unit sphere of an infinite dimensional Banach space, one can always find a finite dimensional subspace of any given dimension on the unit sphere of which the function is almost constant. This motivated the question of whether in this setting one could also pass to an infinite dimensional subspace with the same property.\\

It was only in 1992 when W.T. Gowers proved in  \cite{FINk} that $c_0$, the classical Banach space of real sequences converging to $0$ endowed with the supremum norm, has this property. For the special case of Lipschitz functions defined on the unit sphere of $c_0$ not depending on the sign of the canonical coordinates, that is such that $f(\sum a_ie_i)=f(\sum|a_i|e_i)$ for every $(a_i)$ in the sphere of $c_0$, we can restrict our attention to the positive sphere of $c_0$, the set of elements of the sphere with non-negative canonical coordinates. Gowers associated a discrete structure to a net for the positive sphere of $c_0$,  and proved a partition theorem for the structure that we shall refer to as the $\mathrm{FIN}_k$ Theorem.\\

The $\mathrm{FIN}_k$ Theorem is a generalization of Hindman's Theorem. For a fixed $k\in\N$, $\mathrm{FIN}_k$ is the set of all functions $f:\N\to\{0,1,\ldots,k\}$ that attain the maximum value $k$ and whose support $supp(f)=\{n\in\N: f(n)\neq 0\}$ is finite. Given $f,g\in \mathrm{FIN}_k$ we say that $f<g$ if the support of $f$ occurs before the support of $g$. We consider two operations in $\mathrm{FIN}_k$ defined pointwise as follows:
\begin{itemize}
	\item[(i)] \emph{Sum}: $(f+g)(n)=f(n)+g(n)$ for $f<g$,
	\item[(ii)] \emph{Tetris}: $T:\mathrm{FIN}_k\rightarrow\mathrm{FIN}_{k-1}$. For $f\in\mathrm{FIN}_k$, $(Tf)(x)=\max\{0,f(x)-1\}$.
\end{itemize}
Note that if $f_0<\ldots<f_{n-1}\in \mathrm{FIN}_k$ then $T^{l_0}(f_0)+\ldots+T^{l_{n-1}}(f_{n-1})\in \mathrm{FIN}_k$ as long as one of $l_0,\ldots,l_{n-1}$ is zero. A sequence $(f_i)_{i\in I}$ of elements of $\mathrm{FIN}_k$ with $I=\N$ or $I=n$ for some $n\in\N$ such that $f_i<f_{j}$ for all $i<j\in I$ is called a block sequence. The $\mathrm{FIN}_k$ Theorem states that for any finite coloring $c:\mathrm{FIN}_k\to \{0,1,\ldots,r-1\}$ there exists an infinite block sequence $(f_i)_{i\in\N}$ such that the combinatorial space $\langle f_i\rangle_{i\in\N}$ generated by the sequence $(f_i)_{i\in\N}$,
\[
	\langle f_i\rangle_{i\in\N}=\{T^{l_1}(f_{i_1})+\ldots+T^{l_n}(f_{i_n}): i_1<\ldots<i_n, min\{l_1,\ldots,l_n\}=0\},
\]
is monochromatic.\\
The proof uses Galvin-Glazer methods of ultrafilter dynamics. While the standard modern proof of Hindman's Theorem uses ultrafilter dynamics (see \cite[Ch. 2]{Ramsey_Spaces}), Hindman's theorem was originally proved by constructive methods \cite{Hindman} (see also \cite{Baum}). Until now there is no constructive proof of the $\mathrm{FIN}_k$ Theorem for $k>1$. In this paper we provide a constructive proof of the finite version of the $\mathrm{FIN}_k$ Theorem, namely we prove the following:

 \begin{thm}\label{main1}
 	For all natural numbers $m,k,r$ there exists a natural number $n$ such that for every $r$-coloring of $\mathrm{FIN}_k(n)$, the functions in $\mathrm{FIN}_k$ supported below $n$, there exists a block sequence in $\mathrm{FIN}_k(n)$ of length $m$ that generates a monochromatic combinatorial subspace. 
 \end{thm}
 
Let $g_{k}(m,r)$ be the minimal $n$ given by the theorem. This result follows easily from the infinite version by a compactness argument. However such an argument could not be written in Peano Arithmetic (PA) and gives no information about the bounds of the function $g_k(n,r)$. Sometimes compactness can be used even when there is no inductive proof that can be written in PA. This is the case for the following example of a theorem unprovable in PA found by J. Paris and L. Harrington (see \cite{Paris_Harrington}):\\

	\emph{(PH) For all natural numbers $n,k,r$ there exists a natural number $m$ such that for any $r$-coloring of $[n,m]^k$, the $k$-element subsets of the interval of natural numbers $[n,m]$, there exists a monochromatic set $S$ such that $\#S>\min S$.}\\

The situation is different for the finite $\mathrm{FIN}_k$ Theorem since our proof uses only induction and can be written in PA. In the notation of Theorem \ref{main1}, the bounds we find for $k>1$ are 
\[
	g_k(n,2)\leq f_{4+2(k-1)}\circ f_4(6m-2),
\]
where for $i\in\N$, $f_i$ denotes the $i$-th function in the Ackermann Hierarchy.\\

One could expect that the bounds for the quantitative version of Milman's result about the stabilization of Lipschitz functions on finite dimensional Banach spaces, for the special case of the spaces $\ell_{\infty}^n$, $n\in\N$, (the original statement of the Finite Stabilization Principle can be found in \cite[p.6]{Mil_Schech}), would be comparable to those we find for the finite version of the $\mathrm{FIN}_k$ Theorem. However this is not the case since we find much smaller bounds for the Finite Stabilization Principle in the special case of $\ell_{\infty}^n$- spaces and functions defined on their positive spheres. More precisely, we study in this particular case the following quantitative version of the Finite Stabilization Principle presented in \cite{FDDs}:
 
\begin{thm}\label{finite_stab1}
	For all real numbers $C,\epsilon>0$ and every natural number $m$ there is a natural number $n$ such that for every $n$-dimensional normed space $F$  with a Schauder basis $(\mathbf{x}_i)_{i=0}^{n-1}$, whose basis constant does not exceed $C$, and for every $C$-Lipschitz $f:F\to\R$, there is a block subsequence $(\mathbf{y}_i)_{i=0}^{m-1}$ of $(\mathbf{x}_i)_{i=0}^{n-1}$ so that
	\[
		\mbox{osc}(f\upharpoonright S_{[\mathbf{y}_i]_{i=0}^{m-1}})<\epsilon,
	\]
	where $S_Y$, for $Y$ any Banach space, is the unit sphere of $Y$ $S_Y=S_{[\mathbf{y}_i]_{i=0}^{m-1}}$. And $[\mathbf{y}_i]_{i=0}^{m-1}$ is the vector space generated by the sequence $(\mathbf{y}_i)_{i=0}^{m-1}$.
\end{thm}
Let $N(C,\epsilon,m)$ be the minimal $n$ such that for every $C$-Lipschitz function $f:PS_{\ell_{\infty}^n}\to\R$, there is a block sequence $(\mathbf{y}_i)_{i=0}^{m-1}$ of positive vectors such that $\mbox{osc}(f\upharpoonright PS_{[\mathbf{y}_i]_{i=0}^{m-1}})<\epsilon$. By analysing the proof of Theorem \ref{finite_stab1} in \cite{FDDs} for the especial case of the spaces $\ell_{\infty}^n$, $n\in\N$, we obtain the following upper bound for $N(C,\epsilon,m)$:
 \[
 	N(C,\epsilon,m)\lesssim f_3\left(m^s\cdot\lceil\frac{C}{\epsilon}\rceil^{m^s}\right),
 \]
 where $s=\log(\epsilon/12C)/\log(1-\epsilon/12C)+2$. For fixed $\epsilon$ and $C$, this upper bound is much slower growing than the bound we found for $g_k(m,2)$ for a fixed $k\geq 2$.\\
 
 The paper is organized as follows. In Section \ref{the_proof} we present the proof of Theorem \ref{main1}, in Section \ref{bounds} we obtain the upper bounds given by the arguments in Section \ref{the_proof}. Finally, in Section \ref{finite_stabilization} we introduce the necessary concepts related to the stabilization of Lipschitz functions, and modify the proof of Theorem \ref{finite_stab1} presented in \cite{FDDs} so as to use only finitary arguments. With this modified proof we give upper bounds for $N(C,\epsilon,m)$.
 
 \section{The finite $\mathrm{FIN}_k$ Theorem}\label{the_proof}
	We start by fixing some notation. We denote by $\N$ the set of natural numbers starting at zero and use the Von-Neumann identification of a natural number $n$ with the set of its predecessors, $n=\{0,1,\ldots,n-1\}$. Let $k\in\N$ be given. For $N,d\in\N$	we define the finite version of $\mathrm{FIN}_k$ and its d-dimensional version by:
	\begin{eqnarray*}
	\mathrm{FIN}_{k}(N)&=& \{f\in \mathrm{FIN}_k: \max(\mathrm{supp}(f))<N\}\\
	\mathrm{FIN}_k(N)^{[d]} &=& \{(f_i)_{i<d}|\mbox{ } f_i\in\mathrm{FIN}_{k}(N)\mbox{ and }f_i<f_j\mbox{ for }i<j<d\}.
	\end{eqnarray*}
	
	Where for $f,g\in\mathrm{FIN}_k(N)$, we write $f<g$ when $\max (supp(f))<\min(supp(g))$. The elements of $\mathrm{FIN}_k(N)^{[d]}$ are called block sequences. The combinatorial space $\langle f_i\rangle_{i<d}$ generated by a sequence $(f_i)_{i<d}\in \mathrm{FIN}_k(N)^{[d]}$ is the set of elements of $\mathrm{FIN}_k(N)$ of the form $T^{l_0}(f_{i_0})+\ldots+T^{l_{n-1}}(f_{i_{n-1}})$ where $n\in\N$, $i_0<\ldots<i_{n-1}< d$ and $\min\{l_0,\ldots,l_{n-1}\}=0$. A block subsequence of $(f_i)_{i<d}$ is a block sequence contained in $\langle f_i\rangle_{i<d}$. Just as we defined the $d$-dimensional version of $\mathrm{FIN}_k(N)$, if $(f_i)_{i<l}$ is a block sequence, we define $(\langle f_i\rangle_{i<l})^{[d]}$ to be the collection of block subsequences of $(f_i)_{i<l}$ of length $d$.\\
	
	The following definition is important when coding an element of $\mathrm{FIN}_k$ in a sequence of elements of $\mathrm{FIN}_{k-1}$. Given $\textbf{f}=(f_i)_{i<m}\in \mathrm{FIN}_k^{[m]}$, for
	\[
		g=\sum_{i<m}T^{k-n_i}f_i,
	\]
	we define $supp_k^{\textbf{f}}(g)$ to be the set of all $i<m$ such that $n_i=k$. The cardinality of this set determines the length of the sequence we need in order to code $g$, as we shall describe in detail later on. The proof is by induction on $k$. The starting point is Folkman's Theorem. In the inductive step, the idea is to code an element of $\mathrm{FIN}_k$ in a finite sequence of elements of $\mathrm{FIN}_{k-1}$ and apply the result for $\mathrm{FIN}_{k-1}$ and its higher dimensional versions.\\
	
	We identify canonically $\mathrm{FIN}_1$ with $\mathrm{FIN}$, the collection of finite subsets of $\N$. The case $k=1$ of the finite $\mathrm{FIN}_k$ Theorem, phrased in terms of finite sets and finite unions, is a restatement of Folkman's Theorem. We include a proof of Folkman's Theorem for the sake of completeness and more importantly because we are interested in analysing the corresponding upper bounds. The proof presented here is extracted from \cite[Section 3.4]{RamseyTheory}.
	
	\begin{namedtheorem}[Folkman]
		For all $m,r\in\N$ there exists $N\in\N$ such that for every $c:\mathrm{FIN}(N)\to r$ there exists $(x_i)_{i<m}\in \mathrm{FIN}(N)^{[m]}$ such that $c\upharpoonright \langle x_i\rangle_{i<m}$ is constant. 
	\end{namedtheorem}
	
	It easily follows from the Pigeon-Hole principle that the theorem reduces to the following:
	\begin{lem}\label{N}
		For every $m,r\in\N$ there exists $N\in\N$ such that for all $c:\mathrm{FIN}(N)\to r$ there exists $(x_i)_{i<m}\in \mathrm{FIN}(N)^{[m]}$ such that $c\upharpoonright \langle x_i\rangle_{i<m}$ is $\min$-determined. That is, if $x=\bigcup_{i\in s}x_i, y=\bigcup_{i\in t}x_i$ with $s,t\subseteq m$ such that $\min s=\min t$ then $c(x)=c(y)$. 
	\end{lem}
	We denote by $N(m,r)$ the minimal $N$ given by Lemma \ref{N}. We shall use van der Waerden's Theorem. For $n,r\in\N$, let $W(n,r)$ be the minimal $m$ such that for any $r$-coloring of $m$ there is a monochromatic arithmetic progression of length $n$.
	
	\begin{proof}
		We fix the number of colors $r\in\N$ and proceed by induction on $m$, the length of the desired sequence. The base case $m=1$ is clear, so we suppose the statement holds for $m$ and prove it for $m+1$. By a repeated application of Ramsey's theorem, we fix $N\in\N$ such that given any $r$-coloring of $\mathrm{FIN}(N)$, there exists $A\subseteq N$ of cardinality $W(N(m,r),r)$ such that for all $i<W(N(m,r),r)$, the coloring $c$ is constant on $[A]^i$, the color possibly depending on $i$.\\
		
		Now let $c:\mathrm{FIN}(N)\to r$ be given and let $A\subset N$ be as above with $c\upharpoonright [A]^i$ constant with value $c_i<r$. Define $d:W(N(m,r),r)\to r$ by
		\[
		d(i)=c_i.
		\]
		We use Van der Waerden's Theorem to find $\alpha,\lambda<W(N(m,r),r)$ and $i_0<r$ such that $d\upharpoonright \{\alpha+\lambda j:j<N(m,r)\}$ is constant with value $c_{i_0}$. Let $x_0$ be the set consisting of the first $\alpha$ elements of $A$ and let $y_1<\ldots<y_{N(m,r)}$ be a block sequence of subsets of $A\setminus x_0$ each one of which has cardinality $\lambda$. Note that the combinatorial space generated by $(y_i)_{i<N(m,r)}$ is canonically isomorphic to $\mathrm{FIN}(N(m,r))$, therefore by induction hypothesis there exists a block subsequence $x_1<\ldots<x_m$ of $(y_i)_{0<i\leq N(m,r)}$ such that $c\upharpoonright \langle x_i\rangle_{i=1}^m$ is $\min$-determined.\\
		We shall see that $(x_i)_{i<m+1}$ is the sequence we are looking for. Fix $x,y\in\langle x_i\rangle_{0\leq i\leq m}$ with the same minimum. Suppose first that $x_0\subseteq x$ then also $x_0\subseteq y$ and $\#x=m+\lambda i, \#y=m+\lambda j$ for some $i,j<N(m,r)$. Hence $c(x)=c(y)=c_{i_0}$. Now suppose $x_0\nsubseteq x$ then the same holds for $y$ and consequently $x,y\in\langle x_i\rangle_{i=1}^{m}$. By the choice of $(x_i)_{i=1}^m$ it follows that $c(x)=c(y)$.
	\end{proof}
	We now prove Theorem \ref{main1} in its multidimensional form.
	\begin{thm}\label{main}
	For all $k,m,r,d\in\N$ there exists $n\in\N$ such that for every coloring $c:\mathrm{FIN}_k(n)^{[d]}\to r$ there exists $(f_i)_{i<m}\in \mathrm{FIN}_k(n)^{[m]}$ such that $c\upharpoonright(\langle f_i\rangle_{i<m})^{[d]}$ is constant.   
	\end{thm}
	 Let $g_{k,d}(m,r)$ be the minimal $n$ given by Theorem \ref{main}. We prove the theorem by induction on $k$. Note that if we have the theorem for some $k$ and $d=1$, we can deduce the theorem for $k$ and dimensions $d>0$ using a standard diagonalization argument. We include the dimensions in the statement of the theorem because they play an important role in the proof and because we are interested in calculating upper bounds for $g_{k,d}(m,r)$. 
	\begin{proof}	 
	 The base case $k=1$ in dimension 1 is Folkman's Theorem. Suppose the theorem holds for $k$ and all $m,r,d\in\N$. We work to get the result for $k+1$. We need the following preliminary result:\\
	 
	\begin{claim}\label{barN}
	For all $N,r\in\N$ there exists $\bar{N}$ such that for every $c:\mathrm{FIN}_{k+1}(\bar{N})\to r$ there exists $\textbf{h}=(h_i)_{i<N}\in FIN_{k+1}(\bar{N})^{[N]}$ such that for
	\begin{eqnarray*}
		f &=&\sum_{i<N}T^{k+1-s_i}(h_i)\\
		g &=&\sum_{i<N}T^{k+1-t_i}(h_i),
	\end{eqnarray*}
	$c(f)=c(g)$ whenever $supp_{k+1}^{\textbf{h}}(f)=supp_{k+1}^{\textbf{h}}(g)$, that is, whenever for all $i<N$, $s_i=k+1$ if and only if $t_i=k+1$. 
	\end{claim}
	Let $\bar{N}_{k+1}(N,r)$ be the minimal $\bar{N}$ given by Claim \ref{barN}.
	
	\begin{proof}[Proof of Claim \ref{barN}] 
	Let $N,r\in\N$ be given. By induction hypothesis, let $\bar{N}$ be such that for any sequence of $r$-colorings $(e_i)_{i<N}$ with $e_i:\mathrm{FIN}_k(\bar{N})^{[2i+3]}\to r$, there exists a block sequence $(f_j)_{j<3N}$ such that for each $i<N$, $e_i$ is constant on $(\langle f_j\rangle_{j<3N})^{[2i+3]}$, its value possibly depending on $i$.\\
	Let $c:\mathrm{FIN}_{k+1}(\bar{N})\to r$ be given. Define $U:\mathrm{FIN}_k\to\mathrm{FIN}_{k+1}$ by
	\[
		(Uf)(i)=\begin{cases}
					f(i)+1 &\text{if } f(i)\neq 0\\
					0 &\text{otherwise.}
				\end{cases}	
	\]
	 For each $i<N$ define the coloring $e_i:\mathrm{FIN}_k(\bar{N})^{[2i+3]}\to r$ by
	\[
		e_i((h_j)_{j<2i+3})= c\left(\sum_{j<2i+3} U^{j\mbox{ }mod2}h_j\right),
	\]
	By the choice of $\bar{N}$, we can find a block sequence $(f_j)_{j<3N}$ such that for each $i<N$, $e_i$ is constant on $(\langle f_j\rangle_{j<3N})^{[2i+3]}$. We shall see that the sequence $(h_i)_{i<N}$ defined by $h_i=f_{3i}+Uf_{3i+1}+f_{3i+2}$ for $i<N$ is the sequence we are looking for. Let $g_1,g_2\in\langle h_i\rangle_{i<N}$ be such that $supp_{k+1}^{\textbf{h}}(g_1)=supp_{k+1}^{\textbf{h}}(g_2)$, let $l$ be the cardinality of $supp_{k+1}^{\textbf{h}}(g_1)$. Then we can write $g_1,g_2$ as
	\begin{eqnarray*}
		g_1&=&\sum_{j<2(l-1)+3}U^{j\mbox{ }mod2}w_j\\
		g_2&=&\sum_{j<2(l-1)+3}U^{j\mbox{ }mod2}w'_j
	\end{eqnarray*}
	for some $(w_j)_{j<2(l-1)+3}, (w'_j)_{j<2(l-1)+3}\in (\langle f_j\rangle_{j< 3N})^{[2(l-1)+3]}$. Since $e_{l-1}$ is constant on $(\langle f_j\rangle_{j<3N})^{[2(l-1)+3]}$, it follows that $c(g_1)=c(g_2)$.
	\end{proof}
	
	We now verify that for the case $k+1$, $d=1$ in Theorem \ref{main}, we may take $n=\bar{N}_{k+1}(H,r)$ where $H=g_{1,1}(m,r)$ . Let $c:\mathrm{FIN}_k(n)\to r$ be given. By the choice of $n$ we can find $\textbf{h}=(h_i)_{i<H}$ such that $c\upharpoonright\langle h_i\rangle_{i<H}$ depends only on $supp_{k+1}^{\textbf{h}}$. Let 
	\begin{eqnarray*}
		d:\mathcal{P}(H)&\to& r\\
		x&\mapsto & c\left(\sum_{i\in x} h_i\right)
	\end{eqnarray*}
	By the choice of $H$ we can find $x_0<\ldots <x_{m-1}$ subsets of $H$ such that $d\upharpoonright \langle x_i\rangle_{i<m}$ is constant. For $i<m$ let $f_i=\sum_{j\in x_i}h_j$. Note that for $f\in\langle f_i\rangle_{i<m}$, $supp_{k+1}^{\textbf{h}}(f)$ is a finite union of $x_0,\ldots x_{m-1}$. Therefore $c\upharpoonright\langle f_i\rangle_{i<m}$ is constant.
	\end{proof}
	
	\section{Bounds for the finite $\mathrm{FIN}_k$ Theorem}\label{bounds}
	In this section we calculate upper bounds for the numbers $g_{k,d}(m,r)$ given by the proof in Section \ref{the_proof}. Since we used Ramsey's Theorem and Van der Waerden's Theorem in our arguments, we will need upper bounds for the numbers corresponding to these two theorems. For Ramsey's Theorem, $R_d(m)$ is the minimal $n$ such that if $[n]^d$, the collection of subsets of $n$ of cardinality $d$, is 2-colored then there exists a monochromatic set of cardinality $m$. For Van der Waerden's Theorem, $W(m)$ is the minimal $n$ such that if $n$ is 2-colored, then there exists a monochromatic arithmetic progression of length $m$. For a discussion of Ramsey numbers and Van der Waerden numbers, see \cite[Ch. 4]{RamseyTheory}. It turns out that these numbers grow very rapidly and, in order to deal with such rapidly growing functions, we use the Ackermann Hierarchy. The Ackermann hierarchy is the sequence of functions $f_i:\N\to\N$ defined as follows:
	\begin{eqnarray*}
		f_1(x)&=& 2x\\
		f_{i+1}(x)&=& f_i^{(x)}(1)
	\end{eqnarray*}
	Already the function $f_3$ grows very fast with $f_3(5)=2^{65536}$ a number with nearly 20,000 decimal digits (see \cite[section 2.7]{RamseyTheory}). The function $f_3$ is called TOWER and $f_4$ is called WOW. The Ackermann function is obtained by diagonalization and grows even faster than any $f_i$, $i\in\N$:
	\[
		f_{\omega}(x)=f_x(x)
	\] 
	There is a slight variation of the function TOWER in the Ackermann hierarchy, it is useful to express upper bounds for the Ramsey numbers $R_d(m)$ and the Van der Waerden numbers $W(m)$. The tower functions $t_i(x)$ are defined inductively by
	\begin{eqnarray*}
		t_1(x)&=& x\\
		t_{i+1}(x)&=&2^{t_i(x)}
	\end{eqnarray*}
	We shall use the following well known upper bounds for $R_d(m)$ and $W(m)$:
	\begin{eqnarray}
		R_d(m)&\leq & t_d(c_dm)\label{Ramsey_bound}\\
		W(m)&\leq & 2^{2^{2^{2^{2^{m+9}}}}}=t_6(m+9)\label{VdW_bound} 
	\end{eqnarray}
	Where $c_d$ is a constant that depends on $d$. See \cite[Section 4.7]{RamseyTheory} for a deduction of the bound for $R_d(m)$. The bound for $W(m)$ was found by Gowers in \cite{Gowers_bound}.\\
		
We now analyse the proof of Theorem \ref{main} presented in Section \ref{the_proof} to get some information about the corresponding upper bounds. We consider only 2-colorings and so we will omit the number of colors in the arguments of our functions and write $g_{k,d}(m)$ for $g_{k,d}(m,2)$. That is, $g_{k,d}(m)$ is the minimal $n$ such that for every coloring $c:\mathrm{FIN}_k(n)^{[d]}\to 2$ there exists $(f_i)_{i<m}\in \mathrm{FIN}_k(n)^{[m]}$ such that $c\upharpoonright(\langle f_i\rangle_{i<m})^{[d]}$ is constant. We adopt the same convention for any other numbers defined in the course of the proof of Theorem \ref{main} that have the number of colors as a parameter.\\

In what follows we will analyse each step in the proof presented in Section \ref{the_proof} and refer to the sequence of lemmas and claims presented therein. Recall that in the proof we proceeded by induction on $k$ and in the inductive step from $k$ to $k+1$ we used the inductive hypothesis in several dimensions $d>1$. Therefore we will first find the upper bounds corresponding to the base case $k=1$ in dimension 1, and then describe the diagonalization argument to obtain the result for $k=1$ and dimension 2. The arguments are similar for higher dimensions and so we get upper bounds for $g_{1,d}(m)$, $d>1$. To illustrate how the bounds behave when the value of $k$ increases, we analyse the inductive step in the proof and obtain an upper bound for $g_{2,1}(m)$. The arguments to pass from $k$ to $k+1$ and to increase the dimension are similar for bigger values of $k$ and so we get upper bounds for $g_{k,d}(m)$, $k\geq 2$, $d\geq 1$.\\
	
	We start by finding an upper bound for $N(m)$, the minimal $N\in\N$ such that for all $c:\mathrm{FIN}(N)\to r$ there exists $(x_i)_{i<m}\in\mathrm{FIN}(N)^{[m]}$ such that $c\upharpoonright \langle x_i\rangle_{i<m}$ is $\min$-determined. In the proof of Lemma \ref{N} we had to apply Ramsey's Theorem in dimensions $1,2,\ldots,W(N(m))$ in order to obtain a suitable set of cardinality $W(N(m))$. To easily iterate the upper bound \eqref{Ramsey_bound}, note that for any $i\in\N$ and any given constant $c$, if $x$ is big enough then we have that $t_i(cx)\leq t_{i+1}(x)$. Using these estimates we get the following recursive inequality:
	\begin{eqnarray*}
		N(m+1)&\leq& R_{W(N(m))}\circ\ldots R_2\circ R_1(W(N(m)))\\
		&\leq& t_{l+1}(W(N(m)))\\
		&=& t_{l+6}(N(m)+9)\\
		&\leq& t_{t_6(N(m)+10)}(N(m)+9)\\
		&\leq& t(t_6(N(m)+10)+N(m)+9)\\
		&\leq& t(t_6(N(m)+11 ))\\
		&\leq& t^3(N(m))
	\end{eqnarray*}
	Where $l$ counts the index of the tower function resulting from the iteration of the bound for $R_d(m)$, that is 
	\[
		l=\sum_{i=1}^{t_6(m+9)}i.
	\]
	From the recursive inequality for $N(m)$, we get that
	\begin{eqnarray*}
		N(m)&\leq & (t^3)^m(1)\\
		&\leq & f_4(3m).
	\end{eqnarray*}
	In order to obtain Folkman's Theorem from Lemma \ref{N}, we applied the Pigeon-Hole principle, and so we have that
	\begin{eqnarray}
		g_{1,1}(m)&\leq & N(2(m-1)+1)\\
		&\leq & f_4(6m-3).
	\end{eqnarray}
	We now work to obtain bounds for $g_{1,d}(m)$ for $d>1$. To increase the dimension by 1, we use a standard diagonalization argument which results in upper bounds of the form $g_{1,d}(m)\leq f_{5}^d(7m+2(d-1))$.\\ 
	
	We describe the diagonalization argument we used to obtain Theorem \ref{main} for $k=1$ and $d=2$ and calculate the resulting upper bound for $g_{1,2}(m)$. Let $c:\mathrm{FIN}(N)^{[2]}\to 2$ be given and let us calculate how big should $N$ be in order to ensure the existence of a block sequence of length $m$ generating a monochromatic combinatorial subspace. We define block sequences $S_0,\ldots, S_{p-1}$ and $a_0<\ldots< a_{p-1}$ where $p=g_{1,1}(m)$ with the following properties:
	\begin{itemize}
		\item[(i)] $S_0=\{\{0\},\ldots,\{N-1\}\}$,
		\item[(ii)] $a_j$ is the first element of $S_j$,
		\item[(iii)] For $j>0$, $S_j$ is a block subsequence of $S_{j-1}$ such that for each $x\in\langle a_i\rangle_{i<j}$, the coloring $c_x:\mathrm{FIN}(N\setminus x)\to 2$ of the finite subsets of $N\setminus x$ defined by
		\[
			y\mapsto c(x,y)
		\]
		is constant with value $i_x$ when restricted to $\langle S_j\rangle$,
		\item[(iv)] the sequence $S_{p-1}$ has length 2.
	\end{itemize}
	Each $S_j$, $0<j<p$ can be obtained by a repeated application of Theorem \ref{main} for $k=1$ in dimension 1. Let $S=\{a_j:j<p\}$ and consider the coloring $d:\langle S\rangle\to 2$ defined by
	\begin{equation*}
		x \mapsto i_x
	\end{equation*}
	By the choice of $p$, we can find a block subsequence of $S$ of length $m$ that generates a $d$-monochromatic combinatorial subspace, and by construction this sequence will also generate a $c$-monochromatic combinatorial subspace. Since the total number of refinements to obtain the sequences $(S_j)_{j<p}$ is $2^p-1$, it suffices to start with $N\geq g_{1,1}^{2^p-1}(2)$ and so 
	\begin{equation}
		g_{1,2}(m)\leq g_{1,1}^{2^p-1}(2)
	\end{equation}
	One can prove by induction on $l$ that $g_{1,1}^l(m)\leq f_4^l(6m+l-4)$, so we have that
\begin{eqnarray}
	g_{1,2}(m) &\leq & f_4^{2^p-1}(2^p+7)\label{12first}\\
	&\leq & f_4^{f_4(6m-2)}(f_4(6m-2))\\
	&\leq & f_4^{f_4(6m-2)+1}(6m-2)\\
	&\leq & f_5(f_4(7m))\\
	&\leq & f_5^2(7m).\label{12last}
\end{eqnarray}
	In general the recursive inequality resulting from the diagonalization argument is
	\begin{equation}
		g_{1,d+1}(m)\leq g_{1,1}^{h_d(g_{1,d}(m))}(2),
	\end{equation}
	where $h_d(l)$ for $l\in\N$, is the cardinality of $\mathrm{FIN}_1^{[d]}(l)$. Note that $h_d(l)\leq 2^{ld}$. From calculations like the ones in \eqref{12first}-\eqref{12last}, one gets that for $d\geq 2$
	\begin{equation}
		g_{1,d}(m)\leq f_5^d(7m+2(d-1)).
	\end{equation}
	Now we find an upper bound for $g_{k,1}(m)$, $k>1$. Recall that in the proof of Theorem \ref{main} we proceeded by induction on $k$. In the inductive step from $k$ to $k+1$ we used the higher dimensional versions of the result for $k$. We first found a subsequence where the coloring depends only on $supp_k$, which is the content of Claim \ref{barN}. We then applied Folkman's Theorem to obtain the desired sequence in $\mathrm{FIN}_{k+1}$.\\
	
	We consider the case $k=2$, the calculations for bigger values of $k$ are similar. Let $N\in\N$ be given, in order to establish Claim \ref{barN} for $N$ and $k=2$ we used Theorem \ref{main} for $k=1$ in dimensions $2i+3 (i<N)$. Using the notation in the proof, we see that
	\begin{eqnarray}\label{Nbar2}
		\bar{N}_{2}(N)&\leq & g_{1,2(N-1)+3}\circ\ldots\circ g_{1,5}\circ g_{1,3}(N)\\
		&\leq & g_{1,2(N-1)+3}^N(N)\\
		&\leq & f_5^{2N^2+N}(2N^2+10N-1)\label{powers_g1d}\\
		&\leq & f_6(4N^2+11N-1).
	\end{eqnarray}
	Where in \eqref{powers_g1d} we have used the inequality $g_{1,d}^l(m)\leq(7m+2(d-1)+(l-1)d)$, for $l,d\in\N$, $d>1$. From the final application of Folkman's Theorem we get
	\begin{eqnarray}
		g_{2,1}(m)&\leq &\bar{N}_2(g_{1,1}(m))\\
		&\leq & g_{1,2(g_{1,1}(m)-1)+3}\circ\ldots\circ g_{1,5}\circ g_{1,3}(g_{1,1}(m))\\
		&\leq & f_6(f_4(6m-2)).
	\end{eqnarray}
	For the case $k=2$, the bounds for the higher dimensional numbers we obtain are:
	\begin{equation}
		g_{2,d}(m)\leq f_7^d(f_4(6m-2)+2(d-1)),
	\end{equation}
	In general from the inductive step we get that for any $k\in\N$,
	\begin{eqnarray}
		\bar{N}_{k+1}(N)&\leq& g_{k,2(N-1)+3}\circ\ldots\circ g_{k,5}\circ g_{k,3}(N),\mbox{ and}\label{barNbound}\\
		g_{k+1,1}(m)&\leq& g_{k,2(g_{1,1}(m)-1)+3}\circ\ldots\circ g_{k,5}\circ g_{k,3}(g_{1,1}(m)).\label{k+11bound}
	\end{eqnarray}
	The diagonalization argument used to increase the dimension from $d$ to $d+1$ in the case $k=1$ is similar for bigger values of $k$ so we get that
	\begin{equation}\label{kd+1bound}
		g_{k,d+1}(m)\leq g_{k,1}^{h_{k,d}(g_{k,d}(m))}(2),
	\end{equation}
	where $h_{k,d}(l)$ for $l\in\N$, is the cardinality of $\mathrm{FIN}_k^{[d]}(l)$. Note that $h_{k,d}(l)\leq dl^k$. Using \eqref{barNbound}, \eqref{k+11bound} and \eqref{kd+1bound}, we can carry out similar calculations as the ones presented for the case $k=2$ to obtain:\\
	\begin{eqnarray}
		g_{k,1}(m)&\leq &f_{4+2(k-1)}\circ f_4(6m-2),\\
		g_{k,d}(m)&\leq &f_{5+2(k-1)}^d(f_4(6m-2)+2(d-1)),
	\end{eqnarray}
	where $d>1$. We summarize in the following table the upper bounds we obtain.
	
\begin{center}	
	\begin{tabular}{|c|c|l|}
		\hline
		$k$ & \textit{Dimension} & \textit{Upper bound}\\
		\hline
		\multirow{4}{*}{2} & 1 & $f_6\circ f_4(6m-2)$ \\
 & 2 & $f_7^2(f_4(6m-2)+2)$ \\
 & $\vdots$ & \\
 & d & $f_7^d(f_4(6m-2)+2(d-1))$ \\ \hline
 		\multirow{4}{*}{3} & 1 & $f_8\circ f_4(6m-2)$ \\
 & 2 & $f_9^2(f_4(6m-2)+2)$ \\
 & $\vdots$ & \\
 & d & $f_{9}^d(f_4(6m-2)+2(d-1))$ \\ \hline
 \vdots& & \\  \hline
 \multirow{4}{*}{k} & 1 & $f_{4+2(k-1)}\circ f_4(6m-2)$ \\
 & 2 & $f_{5+2(k-1)}^2(f_4(6m-2)+2)$ \\
 & $\vdots$ & \\
 & d & $f_{5+2(k-1)}^d(f_4(6m-2)+2(d-1))$ \\ \hline
	\end{tabular}
\end{center}	
\section{Bounds for the finite stabilization theorem}\label{finite_stabilization}
As we mentioned in the introduction, Gowers formulated and proved the $\mathrm{FIN}_k$ Theorem to obtain the stabilization of Lipschitz functions on the positive sphere of $c_0$. Let us introduce some notions from Banach space theory in order to talk about stabilization of Lipschitz functions and see how this relates to the $\mathrm{FIN}_k$ Theorem. Given a Banach space $X$ with Schauder basis $(\mathbf{x}_i)_{i\in I}$, $I=\N$ or $I=n$ for some $n\in\N$, for $\mathbf{x}=\sum _{i\in I}a_i\mathbf{x_i}\in X$ we define the \emph{support} $supp(\mathbf{x})$ of $\mathbf{x}$ by $supp(\mathbf{x})=\{i\in I: a_i\neq 0\}$, we say $\mathbf{x}$ is \emph{positive} with respect to $(\mathbf{x}_i)_{i\in I}$ if each $a_i$ is non negative. A sequence of vectors $(\mathbf{y}_i)_{i\in J}$ with $J=\N$ or $J=n$ for some $n\in\N$ is a \emph{block subsequence} of $(\mathbf{x}_i)_{i\in I}$ if each $\mathbf{y}_i$ has finite support and $\max supp(\mathbf{y}_i)<\min supp(\mathbf{y}_{j})$ for every $i<j\in J$. The sequence $(\mathbf{y}_i)_{i\in I}$ is \emph{positive} with respect to $(\mathbf{x}_i)_{i\in I}$ if each $\mathbf{y}_i$ is positive, and it is \emph{normalized} if each $\mathbf{y}_i$ has norm 1. A subspace generated by a positive normalized block sequence is called a \emph{positive subspace}. The \emph{positive unit sphere} of a positive subspace $Y$ of $X$, denoted by $PS_Y$, is the set of positive vectors in the unit sphere of $Y$. We say that a Lipschitz function $f:S_{X}\to\R$ \emph{stabilizes on the positive sphere} if for every $\epsilon>0$ there exists an infinite dimensional positive subspace $Y$ such that
\[
	\mbox{osc}(f\upharpoonright PS_Y)=\sup\{|f(\mathbf{x})-f(\mathbf{y})|:\mathbf{x},\mathbf{y}\in PS_Y\}<\epsilon.
\]
To obtain the stabilization of Lipschitz functions on the whole unit sphere of $c_0$, Gowers used a modification of the combinatorial structure $\mathrm{FIN}_k$ to account for the change of signs, and in this case proved an approximate Ramsey type theorem for it.\\

We say that a space $X$ is \emph{oscillation stable} if every Lipschitz function stabilizes on the unit sphere, that is, for every $\epsilon>0$ there exists an infinite dimensional subspace $Y$ such that
\[
	\mbox{osc}(f\upharpoonright S_Y)=\sup\{|f(\mathbf{x})-f(\mathbf{y})|:\mathbf{x},\mathbf{y}\in S_Y\}<\epsilon.
\]
It turns out that only "$c_0$-like" spaces have this property (see \cite[p. 1349]{Distortion_handbook}). However, given a Lipschitz function defined on the unit sphere of an infinite dimensional Banach space $X$, we can always pass to a subspace of any given finite dimension on the unit sphere of which the oscillation is as small as we want. This was first observed by Milman (see \cite[p.6]{Mil_Schech}). The following theorem gives the quantitative version of this fact for the positive sphere of $\ell_{\infty}^n$-spaces. Recall that the space $\ell_{\infty}$ is the space of bounded sequences of real numbers endowed with the sup norm, the space $\ell_{\infty}^n$ is $\R^n$ endowed with the sup norm.

\begin{thm}\label{finite_stab}
	For all $C,\epsilon>0$ and $m\in\N$ there is $n\in\N$ such that for every $C$-Lipschitz function $f:PS_{\ell_{\infty}^n}\to\R$ there is a positive block sequence $(\mathbf{y}_i)_{i<m}$ so that
	\[
		\mbox{osc}(f\upharpoonright PS_{[\mathbf{y}_i]_{i<m}})<\epsilon.
	\]
\end{thm}
Let $n(C,\epsilon,m)\in\N$ be the minimal $n$ given by Theorem \ref{finite_stab}. This quantitative version is stated and proved in \cite{FDDs} for the sphere of arbitrary finite dimensional Banach spaces. At first, it seemed plausible that Theorem \ref{finite_stab} would suffice to prove the finite $\mathrm{FIN}_k$ Theorem. Given a coloring of $\mathrm{FIN}_k(n)$ for some $n\in\N$, one would have to define a function on a subset of the positive sphere of $\ell_{\infty}^n$ and extend this coloring to the positive sphere of $\ell_{\infty}^n$. Problems arise because we have no control over the Lipschitz constant of the resulting function on the positive sphere of $\ell_{\infty}^n$.\\

It is interesting to see how the bounds found in the previous section for the finite $\mathrm{FIN}_k$ Theorem compare to the bounds for Theorem \ref{finite_stab}. In what follows we shall outline the proof of Theorem \ref{finite_stab} as presented in \cite{FDDs} and calculate the resulting upper bound for the function $n(C,\epsilon,m)$. The argument is organized in two claims. The statement of the first one as we present it here, is slightly different from \cite{FDDs}; in its proof we use Ramsey's Theorem explicitly. We reproduce the argument for the second claim and provide the details that allow us to calculate the upper bounds.\\

Let $(\mathbf{e}_i)_{i\in\N}$ be the standard basis of $c_0$. 

\begin{claim}\label{spreading}
For any $l\in\N$, $C,\epsilon>0$, there exists $\bar{m}\in\N$ such that for any $C$-Lipschitz function $f:PS_{[ \mathbf{e}_{i}]_{i<\bar{m}}}\to\R$, there exists $A\subset \bar{m}$ of cardinality $m$ such that $f\upharpoonright PS_{[\mathbf{e}_i]_{i\in A}}$ is \emph{$\epsilon$- almost spreading}, that is, for any $n_0<\ldots<n_{l-1}, m_0<\ldots<m_{l-1}\in A$ , $l<m$ and any sequence of scalars $(a_i)_{i<l}$ such that $0< a_i\leq 1$ and $\max_i a_i=1$, we have that $|f(\sum_{i<l} a_i\mathbf{e}_{n_i})-f(\sum_{i<l} a_{i}\mathbf{e}_{m_i})|<\epsilon$. 
\end{claim}
Let $\bar{m}(C,\epsilon,m)$ be the minimal $\bar{m}$ given by Claim \ref{spreading}.

\begin{proof}
	Let $\epsilon>0$. Given a $C$-Lipschitz function $f:PS_{[\mathbf{e}_i]_{i<d}}\to\R$, $d\in\N$ and a sequence of scalars $\textbf{a}=(a_i)_{i<l}$ such that $0< a_i\leq 1$ and $\max_i a_i=1$, we define a coloring $c_{f,\textbf{a}}$ of $[d]^l$ as follows: Let $(I_j)_{j<r}$ be a partition of the range of $f$ into intervals of length at most $\epsilon/3$, where $r=\lceil 3C/\epsilon\rceil$. Define $c_{f,\textbf{a}}:[d]^l\to r$ by $c_{f,\textbf{a}}(\left\{n_0,\ldots,n_{l-1}\right\}_{<})=j$ if and only if $f(\sum_{i<l}a_i \mathbf{e}_{n_i})\in I_j$. Note that if $A\subset d$ is homogeneous for $c_{f,\textbf{a}}$ then for any $n_0<\ldots<n_{l-1}, m_0<\ldots<m_{l-1}\in A$, we have that $|f(\sum_{i<l} a_ie_{n_i})-f(\sum_{i<l} a_{i}\mathbf{e}_{m_i})|<\epsilon/3$. We see that $\bar{m}$ should be big enough so that given a $C$-Lipschitz function $f:PS_{[\mathbf{e}_{i}]_{i<\bar{m}}}\to\R$, we can find $A\subset\bar{m}$ of cardinality $m$ that is homogeneous for colorings $c_{f,\textbf{a}}$ with $\textbf{a}$ ranging over some $\epsilon/3C$-nets of $PS_{[\mathbf{e}_i]_{i<l}}$ for $l<m$. For each $l<m$ we use an $\epsilon/3C$-net for $PS_{[\mathbf{e}_i]_{i<l}}$ of cardinality $\lceil3C/\epsilon\rceil^{l-1}$. Hence for each $l<m$ it suffices to use Ramsey's Theorem for $r$-colorings of $l$-tuples, consequently:

	\[
	\bar{m}(C,\epsilon,m)\leq f_3\left(\left(m\Big\lceil\frac{3C}{\epsilon}\Big\rceil^m\right)\cdot\Big\lceil\log_2\frac{3}{\epsilon}\Big\rceil\right)
	\]
\end{proof}
For the next step we need to introduce some notation. We say that $\mathbf{x},\mathbf{y}\in c_{00}$ have the \emph{same distribution}, denoted by $\mathbf{x}\overset{dis}{=}\mathbf{y}$ if $\mathbf{x}=\sum_{i<k}a_i\mathbf{e}_{n_i}$ and $\mathbf{y}=\sum_{i<k}a_i\mathbf{e}_{m_i}$ for some $k\in\N$, $(a_i)_{i<k}\subset\R$, and $n_0<\ldots<n_{k-1}$, $m_0<\ldots<m_{k-1}\in\N$. For $\mathbf{x},\mathbf{y}\in c_{00}$ let
	\[
		dis(\mathbf{x},\mathbf{y})=\inf\{||\mathbf{\bar{x}}-\mathbf{\bar{y}}||_{\infty}:\mathbf{\bar{x}}\overset{dis}{=}\mathbf{x},\mbox{ and }\mathbf{\bar{y}}\overset{dis}{=}\mathbf{y}\}.
	\]
For $0<r<1$ define recursively a sequence of finitely supported vectors $(\mathbf{y}^{(n)}_r)_{n\in\N}$ as follows. Let $\mathbf{y}^{(0)}_r=\mathbf{e}_0$ and assuming $\mathbf{y}^{(n)}_r=\sum_{i<l_n}y^{(n)}_r(i)\mathbf{e}_i$ is already defined we let
	\[	\mathbf{y}^{(n+1)}_r=\sum_{i<l_n}\left(r^{n+1}\mathbf{e}_{3i}+y^{(n)}_r(i)\mathbf{e}_{3i+1}+r^{n+1}\mathbf{e}_{3i+2}\right)
	\]
	(thus $\mathbf{y}^{(1)}_r=(r,1,r,0,\ldots), \mathbf{y}^{(2)}_r=(r^2,r,r^2,r^2,1,r^2,r^2,r,r^2,0,\ldots)$, etc.). Note that $\#supp(\mathbf{y}^{(n)}_r)=3^n$ for $n\in\N$. We shall need the following observation:\\
	\begin{claim}\label{pre_diameter}
		Let $t,s\in\N$, $0<r<1$, and let $(\mathbf{z}_l)_{l<3^t}$ be a block sequence of vectors with the same distribution as $\mathbf{y}^{(st)}_r$. Then for every linear combination $\mathbf{z}=\sum_{l<3^t}r^{\alpha_l}\mathbf{z}_l$ with at least one $\alpha_l=0$, there exists $\bar{\mathbf{z}}\overset{dis}{=}\mathbf{y}^{(st)}_r$ such that for every $j\leq (s-1)t$ and $i\in\N$ such that $\mathbf{z}(i)=r^j$, we have that $\bar{\mathbf{z}}(i)=r^{j+l}$ for some $0\leq l\leq t$. 
	\end{claim}
	\begin{proof}
		Let $s\in\N$ and $0<r<1$. We prove the claim by induction on $t$. For the base case $t=1$, let $(\mathbf{z}_l)_{l<3}$ be a block sequence of vectors distributed as $\mathbf{y}^{(s)}_r$. It is useful to note that 
\begin{eqnarray}\label{mountains}
	 \mathbf{y}^{(s)}_r&\overset{dis}{=}&r\overline{\mathbf{y}^{(s-1)}}+r^2\overline{\mathbf{y}^{(s-2)}}+\ldots r^{(s)}\overline{\mathbf{y}^{(0)}}+\overline{\overline{\overline{\mathbf{y}^{(0)}}}}\\
	 & &+r^{(s)}\overline{\overline{\mathbf{y}^{(0)}}}+\ldots+r^2\overline{\overline{\mathbf{y}^{(s-2)}}}+r\overline{\overline{\mathbf{y}^{(s-1)}}},\nonumber
\end{eqnarray}
where for $i<s$, $\overline{\mathbf{y}^{(i)}}\overset{dis}{=}\overline{\overline{\mathbf{y}^{(i)}}}\overset{dis}{=}\mathbf{y}^{(i)}_r$, $\overline{\overline{\overline{\mathbf{y}^{(0)}}}}\overset{dis}{=}\mathbf{e}_0$ and are such that there is no overlap or gaps between the supports of the terms of the sum. Let $\mathbf{z}=\sum_{l<3}r^{\alpha_l}\mathbf{z}_l$ with at least one $\alpha_l=0$. Since each $\mathbf{y}^{(i)}$, $i<s$ also has an structure as in \eqref{mountains}, it is easy to find a vector distributed as $\mathbf{y}^{(s)}_r$ that satisfies the conclusion of the claim.\\
For the inductive step, let $(\mathbf{z}_l)_{l<3^{t+1}}$ be a block sequence of vectors with the same distribution as $\mathbf{y}^{(s(t+1))}_r$ and let $\mathbf{z}=\sum_{l<3^{t+1}}r^{\alpha_l}\mathbf{z}_l$ with at least one $\alpha_l=0$. It is easy to see that we can write $\mathbf{z}$ as
\[
	\mathbf{z}=\sum_{m<3}r^{\beta_m}\mathbf{w}_m,
\]
where each $\mathbf{w}_m$ is a linear combination of $3^t$ many vectors distributed as $\mathbf{y}^{(s(t+1))}_r$ and at least one $\beta_m=0$. For each $m<3$ let $\bar{\mathbf{z}}_m\overset{dis}{=}\mathbf{y}^{(st)}_r$ be the vector given by the induction hypothesis when applied to the vector obtained from $\mathbf{w}_m$ by restricting to the coordinates with values greater than or equal to $r^{st}$. Let $\bar{\mathbf{z}}=\sum_{m<3}r^{\beta_m}\bar{\mathbf{z}}_m$. We may now apply the claim in the case $t=1$ to the vector obtained by restricting $\bar{\mathbf{z}}$ to the coordinates with value greater than or equal to $r^s$. Let $\bar{\bar{\mathbf{z}}}$ be the vector obtained in this way. It is easy to extend the vector $\bar{\bar{\mathbf{z}}}$ to a vector distributed as $\mathbf{y}^{(s(t+1))}_r$ with the desired property. 
 
	\end{proof}  
\begin{claim}\label{diameter}
	For every $\epsilon>0$ and $m\in\N$ there exists a normalized block subsequence $(\mathbf{z}_i)_{i<m}$ of $(\mathbf{e}_i)$ such that the positive sphere of $[z_i]_{i<m}$ has diameter less than $\epsilon$ with respect to $dis(\cdot,\cdot)$.  
\end{claim}
Let $D(\epsilon,m)$ be the minimal $n$ such that we can find a block subsequence $(\mathbf{z}_i)_{i<m}$ of $(\mathbf{e}_i)_{i<n}$ as in Claim \ref{diameter}.

\begin{proof}
	Let $\epsilon>0$, $m\in\N$ be given. To simplify the notation suppose $m=3^t$ for some $t\in\N$. Let $0<r<1$ be such that $1-r^t<\epsilon/4$ and let $s\in\N$ be such that $r^{(s-1)t}<\epsilon/4$. Take $n=s\cdot t$. Let $\mathbf{z}_0<\ldots<\mathbf{z}_{m-1}$  be distributed as $\mathbf{y}^{(n)}_r$. Let $\mathbf{w}_0,\mathbf{w}_1$ be in the positive sphere of $[z_i]_{i<m}$. Let $\bar{\mathbf{w}}_i$, $i<2$ be linear combinations of the vectors $\mathbf{z}_0,\ldots,\mathbf{z}_{m-1}$ of norm 1, whose coefficients are positive powers of $r$ and such that $||\mathbf{w}_i-\bar{\mathbf{w}}_i||<\epsilon/4$, $i<2$. By Claim \ref{pre_diameter} we have that $dis(\bar{\mathbf{w}}_i,\mathbf{z}_0)<\epsilon/4$, and therefore $dis(\mathbf{w}_0,\mathbf{w}_1)<\epsilon$.\\
	
	We now calculate the upper bound for $D(\epsilon,m)$ given by the proof. From the conditions $1-r^t<\epsilon/4$ and $r^{(s-1)t}<\epsilon/4$ we get
	\begin{eqnarray*}
		s&>&\frac{\vert\log\left(\epsilon/4\right)\vert}{\vert t\log r\vert}+1\\
		\vert t\log r\vert&<&\vert\log(1-\epsilon/4)\vert,
	\end{eqnarray*}
	so $s>\frac{\vert\log(\epsilon/4)\vert}{\vert\log(1-\epsilon/4)\vert}+1$ and we have that
	\[
		D(\epsilon,3^t)\leq 3^{(s+1)\cdot t}\leq (3^t)^{\frac{\vert\log(\epsilon/4)\vert}{\vert\log(1-\epsilon/4)\vert}+2}.
	\]
\end{proof}
To see how Theorem \ref{finite_stab} follows from Claims \ref{spreading} and \ref{diameter}, and obtain the resulting upper bound for $n(C,\epsilon,m)$, let $C,\epsilon>0$, $m\in\N$ be given. Suppose $m=3^t$ for some $t\in\N$. Let 
\[
n=\bar{m}(C,\epsilon/3,D(\epsilon/3C,m)).
\]

Let $f:\ell_{\infty}^n\to\R$ be $C$-Lipschitz. By Claim \ref{spreading}, we can find $A\subset n$ of cardinality $D(\epsilon/3C,m)$ such that $f$ is $\epsilon/3$-almost spreading on $PS_{[\mathbf{e}_i]_{i\in A}}$. By Claim \ref{diameter} we can find a block subsequence $(\mathbf{z}_i)_{i<m}$ of $(\mathbf{e}_i)_{i\in A}$ such that $PS_{[\mathbf{z}_0,\ldots,\mathbf{z}_{m-1}]}$ has diameter less than $\epsilon/3C$ with respect to $dis(\cdot,\cdot)$.\\

Let $\mathbf{y}_i\in PS_{[\mathbf{z}_0,\ldots,\mathbf{z}_{m-1}]}$, $i=0,1$. We can find $\mathbf{\bar{y}}_i\in PS_{[\mathbf{z}_0,\ldots,\mathbf{z}_{m-1}]}$ such that $\mathbf{\bar{y}}_i\overset{dis}{=}\mathbf{y}_i$, $i=0,1$ and $\vert\vert \mathbf{\bar{y}}_0-\mathbf{\bar{y}}_1\vert\vert_{\infty}<\epsilon/3C$. Hence we have that 
\begin{eqnarray*}
	\vert f(\mathbf{y}_0)-f(\mathbf{y}_1)\vert&\leq& \vert f(\mathbf{y}_0)-f(\mathbf{\bar{y}}_0)\vert+\vert f(\mathbf{\bar{y}}_0)-f(\mathbf{\bar{y}}_1)\vert+\vert f(\mathbf{y}_1)-f(\mathbf{\bar{y}}_1)\vert\\
	&<&\epsilon.
\end{eqnarray*}
Therefore 
\[
	n(C,\epsilon,3^t)\leq\bar{m}(C,\epsilon/3,D(\epsilon/3C,3^t))<f_3\left(3^{t\cdot s}\lceil\frac{9C}{\epsilon}\rceil^{3^{t\cdot s}}\lceil\log_2\frac{9}{\epsilon}\rceil\right),
\]
where $s=\lceil\frac{\log(\epsilon/12C)}{\log(1-\epsilon/12C)}\rceil+2$.

\section{Conclusions}
As far as we know there is no proof of the infinite $\mathrm{FIN}_k$ Theorem that avoids the use of idempotent ultrafilters. The proof we present of the finite version cannot be adapted to the infinite case. This is because when proving the result for $k+1$, we have to know how many dimensions of the inductive hypothesis we need, and this number depends on the desired length of the homogeneous sequence.\\

We found upper bounds for the Finite Stabilization Theorem in the special case of the spaces $\ell_{\infty}^n$ that grow much slower than the upper bounds we have for the finite $\mathrm{FIN}_k$ Theorem, this suggests that the $\mathrm{FIN}_k$ Theorem is stronger than this special case of the Finite Stabilization Theorem. To make this comparison precise and also because it is interesting in its own right, we still have to find lower bounds for the functions $g_k(n), k\in\N$. This would amount to finding for any given $l\in\N$, a bad coloring of $\mathrm{FIN}_k(N)$ for some $N$, for which there is no sequence of length $l$ generating a monochromatic combinatorial subspace. In this direction it would also be interesting to find a way for stepping up lower bounds for a given $k\in\N$ to bigger values of $k$.\\
 
\emph{Acknowledgments.} We would like to thank Justin Moore for his helpful suggestions. The author is grateful to Jordi Lopez-Abad for sharing his insights throughout the development of this project and for reading earlier versions of this article. His comments greatly improved the content and presentation. We would also like to thank the Fields Institute for their hospitality during the Fall of 2012 when the main part of this project was developed.

\bibliography{C:/Users/Diana/Documents/texfiles/mybib}

\def\cprime{$'$}
\begin{thebibliography}{10}

\bibitem{Baum}
J.~E. Baumgartner.
\newblock A short proof of {H}indman's theorem.
\newblock {\em J. Combinatorial Theory Ser. A}, 17:384--386, 1974.

\bibitem{FINk}
W.~T. Gowers.
\newblock Lipschitz functions on classical spaces.
\newblock {\em European J. Combin.}, 13(3):141--151, 1992.

\bibitem{Gowers_bound}
W.~T. Gowers.
\newblock A new proof of {S}zemer\'edi's theorem.
\newblock {\em Geom. Funct. Anal.}, 11(3):465--588, 2001.

\bibitem{RamseyTheory}
Ronald~L. Graham, Bruce~L. Rothschild, and Joel~H. Spencer.
\newblock {\em Ramsey theory}.
\newblock John Wiley \& Sons Inc., New York, 1980.
\newblock Wiley-Interscience Series in Discrete Mathematics, A
  Wiley-Interscience Publication.

\bibitem{Hindman}
N.~Hindman.
\newblock Finite sums from sequences within cells of a partition of {$N$}.
\newblock {\em J. Combinatorial Theory Ser. A}, 17:1--11, 1974.

\bibitem{Mil_Schech}
Vitali~D. Milman and Gideon Schechtman.
\newblock {\em Asymptotic theory of finite-dimensional normed spaces}, volume
  1200 of {\em Lecture Notes in Mathematics}.
\newblock Springer-Verlag, Berlin, 1986.
\newblock With an appendix by M. Gromov.

\bibitem{FDDs}
E.~Odell, H.~P. Rosenthal, and Th. Schlumprecht.
\newblock On weakly null {FDD}s in {B}anach spaces.
\newblock {\em Israel J. Math.}, 84(3):333--351, 1993.

\bibitem{Distortion_handbook}
Edward Odell and Th. Schlumprecht.
\newblock Distortion and asymptotic structure.
\newblock In {\em Handbook of the geometry of {B}anach spaces, {V}ol.\ 2},
  pages 1333--1360. North-Holland, Amsterdam, 2003.

\bibitem{Paris_Harrington}
J.~Paris and L.~Harrington.
\newblock A mathematical incompleteness in {P}eano {A}rithmetic.
\newblock pages 1133--1142, 1977.

\bibitem{Ramsey_Spaces}
S.~Todor\v{c}evi\'c.
\newblock {\em Introduction to {R}amsey spaces}, volume 174 of {\em Annals of
  Mathematics Studies}.
\newblock Princeton University Press, Princeton, NJ, 2010.

\end{thebibliography}
\end{document}